\documentclass[12pt]{amsart}
\usepackage{amsmath}
\usepackage{amssymb}
\usepackage{mathrsfs}
\usepackage{graphicx}
\usepackage{tikz,float,paralist}
\usepackage{hyperref}

\makeatletter
\DeclareFontFamily{U}{tipa}{}
\DeclareFontShape{U}{tipa}{m}{n}{<->tipa10}{}
\newcommand{\arc@char}{{\usefont{U}{tipa}{m}{n}\symbol{62}}}%

\newcommand{\arc}[1]{\mathpalette\arc@arc{#1}}

\newcommand{\arc@arc}[2]{%
  \sbox0{$\m@th#1#2$}%
  \vbox{
    \hbox{\resizebox{\wd0}{\height}{\arc@char}}
    \nointerlineskip
    \box0
  }%
}
\makeatother

\newtheorem{thm}{Theorem}[section]

\newtheorem{lem}{Lemma}[section]
\newtheorem{prop}{Proposition}[section]
\newtheorem{cor}{Corollary}[section]
\theoremstyle{definition}
\newtheorem{dfn}{Definition}[section]

\theoremstyle{remark}
\newtheorem{rem}{Remark}
\numberwithin{equation}{section}

\title[Convex curves in 2d Alexandrov spaces]{Comparison and Rigidity Theorems for geodesic curvatures in two dimensional Alexandrov spaces}
\author{Le Ma, John Man Shun Ma}
\date{\today}

\begin{document}

\begin{abstract}
In this work, we study geodesic curvature of the boundary of a two dimensional Alexandrov space of curvature bounded below (CBB). We prove several comparison and globalization theorems for the geodesic curvature, generalizing the known results for curves in space of curvature bounded above (CBA) \cite{AB}. We also prove a rigidity theorem for boundary with corners and geodesic curvature lower bound. This generalizes the known rigidity result in \cite{GP} in 2d.
\end{abstract}

\maketitle

\section{Introduction}
For any smooth curve in a Riemannian manifold, its geodesic curvature is the second order quneity which measures how the curve differs from a geodesic. For curves in general length spaces, there are several generalizations of geodesic curvatures. In \cite{AB}, Alexander and Bishop study the arc/chord curvature $\kappa$ and osculating curvature $\chi$ of a curve in a CBA space. Definitions are recalled in Section 2. Under assumptions on upper bounds of either $\kappa$, $\chi$, they obtain several comparison, globalization and rigidity theorems, which generalizes known results in the smooth setting. 

In this work, we study the corresponding problems for curves in a two dimensional Alexandrov space $X$ of curvature bounded below by $K$, where $K\ge 0$. It is known that $X$ is a topological surface, possibly with boundary $\partial X$. We assume that $X$ is topologically a disk with boundary $\partial X \cong \mathbb S^1$ and prove the following results. 

First, we prove a globalization theorem for the arc/chord curvature (Theorem \ref{thm globalization for the arc/chord curvature}). Recall that the Globalization Theorem \cite[Section 3]{BGP} of CBB($K$) says that the triangle comparison holds for all geodesic triangles in $X$ and $X$ admits a diameter bound. 
Roughly speaking, in theorem \ref{thm globalization for the arc/chord curvature} we  prove that if The boundary $\partial X$ has an almost everywhere pointwise lower bound on the (upper) arc/chord curvature, then $\partial X$ admits a global arc/chord curvature lower bound. 

The corresponding globalization result for curves in general CBA($K$) spaces is proved in \cite{AB}. In the CBB case, we need to assume that $X$ is of two dimensional: unlike the CBA($K$) situation, pointwise curvature lower bounds of geodesic curvatures give little control of the global behavior (see remark \ref{rem globalization fails in high d}). We note that the Reshetnyak Majorization Theorem \cite[9.L]{AKP}, which is used essentially in \cite{AB} to obtain comparison and globalization theorems for curves with geodesic curvature upper bound in CBA space, is not available in the CBB setting.

Second, we derive an (almost everywhere) point-wise inequality between the arc/chord curvature and the osculating curvature of $\partial X$ (Theorem \ref{thm construction of F}). When $X$ is a domain in $\mathbb R^2$ with convex boundary $\partial X$, the arc/chord curvature $\kappa$ and the osculating curvature $\chi$ of $\partial X$ exist almost everywhere and are equal (see Appendix B). In \cite[Corollary 3.4]{AB}, it is shown that for a rectifiable curve in a CBA space, the (upper) arc/chord curvature upper bound $\overline \kappa \le \kappa$ implies the (upper) osculating curvature upper bound $\overline\chi \le \kappa$. See \cite[Theorem 5.3]{AB1} for similar results for finite dimensional CBB space. To the authors knowledge, Theorem \ref{thm construction of F} seems to be the only (almost everywhere) pointwise inequality available in the literature. 

The third theorem (Theorem \ref{thm length bounds of gamma with arc/chord lower bound and rigidity}) is a rigidity theorem for boundary $\partial X$ with corners. When $X$ is a smooth Riemannian surface with positive Gaussian curvature and $\partial X$ has positive geodesic curvature, Toponogov \cite{Toponogov} derives a length upper bound for the smooth boundary (see also \cite[p.297]{K} and \cite[Theorem 4]{HW}). It had since been generalized in the Alexandrov setting when $\partial X$ has positive integral geodesic curvature (swerve) \cite{Borisenko}.

In higher dimension, Petrunin \cite{P} proves that if $X$ is an $n$-dimensional Alexandrov space with curvature $\ge 1$, then the boundary satisfies $\mathcal H^{n-1} (\partial X) \le \mathcal H^{n-1}(\mathbb S^{n-1})$. In \cite{GP}, Grove and Peterson prove that when equality holds, then $X$ is isometric to a Alexandrov lens (the intersection of two hemi-spheres in $\mathbb S^n$). See \cite{DK}, \cite{GL1}, \cite{GL2} for more related results on radius estimates and rigidity, and \cite{HW} for similar rigidity result for the Ricci curvature. In Theorem 6.1, we prove a length upper bound for $\partial X$, assuming that $\partial X$ has arc/chord curvature $\ge \kappa$ and a corner with turning angle $\theta$. We also prove the rigidity that when the length upper bound is attained, then $X$ is isometric to a $\kappa$-lens (the intersection of two $\kappa$ disks in $S_K$). 


The organization of this paper is as follows. In section \ref{section: background}, we discuss the background in two dimensional Alexandrov spaces and different notions of geodesic curvatures. In section \ref{section representing convex curves in CBB(K)}, we identify $\partial X$ with a convex curve $\tilde\gamma$ in $S_K$. This is used in section \ref{section globalization for the arc/chord} to prove the globalization theorem for the arc/chord curvature. The inequality between the arc/chord and the osculating curvature is proved in section \ref{section kappa ge chi}. In section \ref{section rigidity}, we prove the rigidity theorem \ref{thm length bounds of gamma with arc/chord lower bound and rigidity}.

John Man Shun Ma received funding from the National Natural Science Foundation of China (General Program. Funding no.:12471053).

\section{Background} \label{section: background}
In this section, we first review the basics of Alexandrov spaces. The main references are \cite{BBI}, \cite{AKP} and \cite{AlexandrovselectedworksII}. 

Let $X$ be a complete length space. The distance between $p, q\in X$ is denoted $|pq|_X$. A geodesic from $p$ to $q$ is denoted $[pq]$. Let $K\ge 0$, and let $S_K$ be the 2-dimensional model space with constant sectional curvature $K$. Hence $S_K$ is 2-sphere with radius $1/\sqrt K$ when $K>0$ and $S_0= \mathbb R^2$. Let $p, q\in S_K$. We use $|pq|_K$ or just $|pq|$ to denote the distance between $p, q$ in $S_K$. 

A simple closed curve in $S_K$ is called convex if it bounds a geodesically convex region $\Omega$. $\Omega$ is called the interior of the curve. When $K>0$, any closed convex curve lies in one of the hemisphere of $S_K$. 

A circle (resp. arc, minor arc, major arc) in $S_K$ with constant geodesic curvature $\kappa$ in $S_K$ is called a $\kappa$-circle (resp. arc, minor arc, major arc) in $S_K$. Up to rigid motion of $S_K$, for any $\kappa \ge 0$ there is only one $\kappa$-circle. The convex region bounded by a $\kappa$-circle in $S_K$ is called a $\kappa$-disk and is denoted $\mathbb D_{ \kappa}$. 

Following \cite{GP}, the intersection of two $\kappa$-disks in $S_K$ is called a $\kappa$-lens. More precisely, let $\beta$ be the interior angle between the two boundaries of the two $\kappa$-disks. The turning angle of the $\kappa$-lens is defined as $\theta = \pi-\beta$, and the $\kappa$-lens is denoted $\mathcal L_{\kappa, \theta}$. Note that $\mathcal L_{\kappa, 0}$ is just a $\kappa$-disk. By direct calculations (see Appendix C), one has 
\begin{equation} \label{eqn perimeter of kappa lens} 
    L(\partial \mathcal L_{\kappa, \alpha}) = \frac{4}{\sqrt{K+\kappa^2}}\cdot\arcsin\sqrt{1-\frac{\kappa^2(1-\cos\theta)}{K(1+\cos\theta)+2\kappa^2}}.
\end{equation}
Note that (\ref{eqn perimeter of kappa lens}) reduces to $L(\partial \mathbb{D}_{K,\kappa}) = 2\pi/\sqrt{K+\kappa^2}$ when $\theta =0$.



\subsection{Osculating curvature and arc/chord curvatures}
The following definitions are taken from \cite{AB}. Let $(X, d)$ be a metric space and let $p, q, m\in X$. Let $\triangle_K \tilde q\tilde q\tilde m$ be its comparison triangle in $S_K$ (if exists). Up to rigid motions of $S_K$, its comparison triangle $\Delta_K \tilde p \tilde q\tilde m$ lies in a unique $\kappa$-circle in $S_K$. Define $\chi_K (p , q, m) = \kappa$.  



\begin{dfn} \label{dfn of osculating curvature}
    Let $\gamma : I \to X$ be a curve and $q=\gamma(t_0)\in \gamma$. The lower and upper osculating curvature of $\gamma$ at $q$ are respectively
    \begin{equation}
        \underline{\chi}_{\gamma}(q) =\liminf_{p, m\to q} \chi_K (p, q, m), \ \      \overline{\chi}_{\gamma}(q) =\limsup_{p, m\to q} \chi_K (p, q, m),
    \end{equation}
    where the limit is taken so that $p, m\to q$ from opposite sides of $q$. That is, $p=\gamma(s_1)$, $m=\gamma(s_2) = q$ and $s_1<t_0<s_2$. If $\underline{\chi}_\gamma(q)=  \overline{\chi}_\gamma(q)$, the limit is denoted as $\chi_\gamma(q)$ and is called the osculating curvature of $\gamma$ at $q$.  
\end{dfn}

\begin{rem}
    In \cite{AB1}, \cite{GP}, the authors define the base-angle/chord curvature for the boundary $\partial X$ of a finite dimensional CBB($K$) space $X$. When $X$ is $2$-dimensional and $\gamma$ is a parametrization of $\partial X$, we show in Appendix B that this is closely related to the osculating curvature of $\gamma$. 
\end{rem}

When $\gamma$ is rectifiable, one defines the arc/chord curvature as follows: let $\arc I$ be any arc in $\gamma$ with endpoints $p, q$. When $K>0$ we require that $|\arc I| + |pq|_X \le 2\pi /\sqrt K$. Up to rigid motions of $S_K$, there is a unique $\kappa$-arc $\arc{\tilde I}$ in $S_K$ with endpoints $\tilde p, \tilde q$ such that $|\arc I| = |\arc{\tilde I}|$ and $|pq|_X = |\tilde p \tilde q|$. Define 
\begin{equation}
    \kappa _K(\arc {I}) = \kappa.
\end{equation}

\begin{dfn} \label{dfn of arc/chord curvature}
    Let $\gamma$ be a rectifiable curve in $X$ and $q\in \gamma$. Define the lower and upper arc/chord curvature of $\gamma$ at $q$ as 
    \begin{equation}
        \underline{\kappa}_\gamma (q) = \liminf_{q\in \arc I, |\arc I|\to 0} \kappa _K(\arc I), \ \  \overline{\kappa}_\gamma (q) = \limsup_{q\in \arc I, |\arc I| \to 0} \kappa _K(\arc I),
    \end{equation}
    where the limit is taken over all subarcs $\arc I$ of $\gamma$ containing $q$. If $\underline{\kappa}_\gamma(q) = \overline{\kappa}_\gamma (q)$, the limit is denoted as $\kappa_\gamma(q)$ and is called the arc/chord curvature of $\gamma$ at $q$. 
\end{dfn}

In this work, we are interested in curves with lower bounds on osculating or arc/chord curvature, which includes almost everywhere lower bounds on $\chi_\gamma(s)$, $\kappa_\gamma(s)$, and the following global lower bound: 

\begin{dfn} \label{dfn globally curvature lower bound}
Let $\gamma$ be a curve in $X$. We say that $\chi_{\gamma, K} \ge\kappa$ globally, if for any three points $p, q, m \in \gamma$ such that $|\triangle pqm |\le 2\pi/\sqrt K$ when $K>0$, we have $\chi(p, q, m) \ge \kappa$. If $\gamma$ is rectifiable, we say that $\kappa_{\gamma, K} \ge \kappa$ globally if for any subarc $\arc I$ of $\gamma$ with endpoints $p, q$ such that $|\arc I| + |pq|_X \le 2\pi/\sqrt K$ when $K>0$, we have $\kappa (\arc I) \ge \kappa$. 
\end{dfn}

\begin{rem} \label{rem globalization fails in high d}
The following examples illustrate that, in general, pointwise lower bound on geodesic curvature does not imply global bound as in definition \ref{dfn globally curvature lower bound}: The helix $\gamma(t) = (\cos t, \sin t, t)$ has constant geodesic curvature $\kappa_0>0$. However, it does not satisfy $\chi_{\gamma, 0} \ge r$ for any $r >0$ since it passes through a straight line infinitely many times. Similarly, the curve $\gamma (t) = (\cos t, \sin t, t^{-1})$ for $t >0$ admits a poisitive geodesic curvature lower bound, but it does not satisfy $\kappa_{\gamma, 0} \ge r$ for any $r >0$. 
\end{rem}

\begin{rem} \label{rem under over kappa independent of K}
    In general, the lower arc/chord curvature is given by \cite[Proposition 2.2]{AB}:
    $$ \underline\kappa_{\sigma} (q) = \liminf_{\arc I} \sqrt{24\frac{\ell-r}{r^3}},$$
    where the limit is taken over all subarc $\arc I$ of $\sigma$ containing $q$ with arc-length $\ell$ and chord length $r$. Similarly for $\overline{\kappa}_{\gamma}(q)$. As a result, $\underline\kappa$ and $\overline\kappa$ are independent of the model space $S_K$ chosen.
\end{rem}

\begin{rem} \label{rem under over chi independent of K}
    Using the formula (\ref{eqn dfn osculating curvature in R^2}) for osculating curvature of triangles in $\mathbb R^2$ and a calculation relating the arc length of two points $p, q \in S_K$ and $|p-q|$ (calculated in $\mathbb R^3 \supset S_K$), one sees that $\underline\chi_\gamma$ and $\overline\chi_\gamma$ are also independent of $S_K$ chosen for any $K\ge 0$. 
\end{rem}

\subsection{Convex curve in \texorpdfstring{$\mathbb R^2$}{R^2}, support function}
Let $\sigma$ be a convex curve in $\mathbb R^2$. For almost all $q\in \sigma$, both the osculating curvature and the arc/chord curvature at $q$ exists and are equal. If $\sigma = \sigma(s)$ is an arc-length parametrization of $\sigma$, then for almost all $s$, the osculating and arc/chord curvature at $\sigma(s)$ equals $\theta'(s)$, there $\theta$ is the angle $\sigma'(s)$ makes with the $x$-axis (see Appendix A). 




For each $\theta \in [0,2\pi]$, let $e(\theta) = (\cos\theta, \sin\theta)$. The support function of $\sigma$ is given by
\begin{equation} \label{eqn dfn of h}
h : [0,2\pi]\to \mathbb R, \ \    h(\theta) = \sup_{p\in \sigma}  \ p \cdot e(\theta). 
\end{equation}
If $\theta'(s) \ge \kappa>0$ a.e. $s\in [0, \ell]$, then $h$ is $C^{1,1}$ and satisfies
\begin{equation} \label{eqn h'' + h = s'}
    h''+ h = s'(\theta), \ \ \text{ for almost all } \theta \in [0, 2\pi]. 
\end{equation}
Here $s = s(\theta)$ is the left inverse of $s\mapsto \theta (s^\pm)$.

The support function is very useful for studying the curvature properties of convex subsets. For example, the following is essentially proved in \cite[Lemma 1]{Borisenko}. 

\begin{lem} \label{lem kappa_sigma ge kappa_0 implies sigma lies inside a kappa_0 circle}
Let $\sigma$ be a closed convex curve in $S_K$ such that $\kappa_\sigma(q) \ge \kappa$ for almost every $q\in \sigma$. Then for each $p\in \sigma$ and any supporting line of $\sigma$ at $p$, $\sigma$ is contained in the region bound by a $\kappa$-curve which is tangential to the supporting line at $p$.  
\end{lem}

In this work, the support function is used to prove the following

\begin{thm} \label{prop kappa_gamma ge g implies kappa_gamma ge g globally for curves in R^2}
    Let $\sigma$ be a closed convex curve in $\mathbb R^2$ such that $\kappa_\sigma(q) \ge \kappa >0$ for a.e. $q\in \sigma$. Then $\kappa_{\sigma, 0} \ge \kappa$ globally. 
\end{thm}

\begin{proof}
Let $\arc I$ be any subarc of $\sigma$ with endpoints $p,q$. Let $\sigma : [0, \ell_0] \to \arc I$ be the arc-length parametrization of $\arc I$ from $p$ to $q$, where $\ell_0 = |\arc I|$. We may assume that $p$ is the origin and $\theta(0^+) = 0$. We will show that $\kappa_0 (\arc I) \ge \kappa$. 

Let $\tilde I$ be a $\kappa$-arc with $|\tilde I| = |\arc I|$ and endpoints $\tilde p, \tilde q$. Let $\tilde r = |\tilde p\tilde q|$. Then
\begin{equation}
    \tilde r = \frac{2}{\kappa} \sin\left(\frac{\kappa\ell_0}{2}\right)
\end{equation}
and $\kappa_0 (\arc I) \ge \kappa$ if $r=|pq| \le \tilde r$. To prove this, let $h$ be the support function of $\sigma$ with respect to the origin $p=0$. Hence $r=r(\theta_0) = |\sigma(\theta_0)|$, where $\sigma(\theta) =\sigma (s(\theta))$. Using (\ref{eqn dfn of h}), one has 
$$h(\theta) = \sigma(\theta) \cdot e(\theta), \ \  h'(\theta) = \sigma(\theta) \cdot e(\theta+ \pi/2).$$ 
Thus 
\begin{equation}
    r^2(\theta) = h(\theta)^2 + (h'(\theta))^2.
\end{equation}
On the other hand, from (\ref{eqn h'' + h = s'}) and $h(0) = h'(0)  =0$, one obtains
\begin{align*}
 h(\theta) = \int_0^\theta s' (\theta-\varphi) \sin\varphi \, d\varphi, \ \ \ 
 h'(\theta) = \int_0^\theta s' (\theta-\varphi) \cos\varphi \, d\varphi. 
\end{align*}
Write $u(\varphi) = s'(\theta-\varphi)$, we have 
\begin{equation}
    h^2+ (h')^2 = \left(\int_0^{\theta} u \cos \varphi d\varphi\right)^2+ \left(\int_0^{\theta} u \sin \varphi d\varphi\right)^2
\end{equation}
To estimates $h^2(\theta_0)+(h' (\theta_0))^2$ we frame it as the following optimization problem (note that the arc/chord curvature equal $\theta'$ almost everywhere): 
\begin{equation}
    \text{ Maximize } \mathcal F(u) = \left(\int_0^{\theta_0} u \cos \varphi d\varphi\right)^2+ \left(\int_0^{\theta_0} u \sin \varphi d\varphi\right)^2
\end{equation}
among all $u \in \mathscr C$, where 
\begin{equation}
\mathscr C = \left\{ u\in L^\infty (0,\theta_0) : 0\le u\le \kappa^{-1}, \ \ \int_0^{\theta_0} u = \ell_0\right\}.
\end{equation}
Let $\tilde{\mathcal F} : \mathscr C\to \mathbb R^2$ be the map 
\begin{equation}
\tilde{\mathcal F}(u) = \left( \int_0^{\theta_0} u(\varphi) \cos\varphi d\varphi, \int_0^{\theta_0} u(\varphi) \sin \varphi d\varphi\right).
\end{equation}
Then $\mathcal F(u) = |\tilde{\mathcal F}(u)|^2$ and 
\begin{align*}
\max_{u\in \mathscr C} \mathcal F (u) &= \max_{0\le \psi \le 2\pi} \max_{u\in \mathscr C} \left| \tilde{\mathcal F}(u) \cdot (\cos\psi, \sin\psi)\right|^2\\
&= \max_{0\le \psi \le 2\pi } \max_{u\in \mathscr C} \left| \int_0^{\theta_0} u(\varphi) \cos (\psi-\varphi) d\varphi\right|^2
\end{align*}
For each fixed $\psi$, it is clear that the maximum 
$$\max_{u\in \mathscr C} \left| \int_0^{\theta_0} u(\varphi) \cos (\psi-\varphi) d\varphi\right|^2$$
is attained when the mass of $u$ stays as close to $\psi$ as possible. Hence, the maximum is attained when $\varphi \in [\kappa\ell_0/2,2\pi - \kappa\ell_0/2]$ and $u = \kappa^{-1}\chi_{I_\varphi}$, where $I_\varphi = [\varphi - \kappa\ell_0/2,\varphi + \kappa\ell_0 /2]\subset [0,\theta_0]$. As a result, 
\begin{align*}
\max_{u\in \mathscr C} \mathcal F(u) &= \frac{4}{\kappa^2}\left(\int_0^{\kappa\ell_0/2} \cos\varphi d\varphi\right)^2 \\
&= \frac{4}{\kappa^2} \sin^2 \left( \frac{\kappa\ell_0}{2}\right) \\
&= \tilde r^2. 
\end{align*}
Thus $r\le \tilde r$ and this completes the proof of the theorem. 
\end{proof}

Recall that in our definition of the osculating curvature at $q$, the limit is taken over triangles where $q$ is one of the vertices. In the following proposition, it is shown that for almost all points in $\sigma\subset \mathbb R^2$, the osculating curvature at $q$ can be evaluated using triangles with vertices lying on both sides of $q$ and are not too thin. The proof is given in Appendix A.


\begin{prop} \label{prop calculate chi by not so thin triangles}
    Let $\sigma$ be a closed convex curve in $\mathbb R^2$. Then for almost all $q\in \sigma$, the osculating curvature at $q$ exists and 
\begin{equation}
\chi_\sigma(q) = \lim \chi_K(p_1, p_2, p_3),
\end{equation}
where the limit is taken over all $p_1, p_2, p_3\in \sigma$ such that 
\begin{itemize}
\item [(i)] $p_1, p_2, p_3 \to q$, 
\item [(ii)]$p_2, q\in \arc{p_1p_3}$, and 
\item [(iii)] the arc lengths $|\arc{p_1p_2}|$, $|\arc{p_2p_3}|$ and $|\arc{p_1p_3}|$ satisfy
\begin{equation} \label{eqn |arc p_1p_2| ge 1/3 |arc p_1p_3|}
|\arc{p_1p_2}|, |\arc{p_2p_3}|\ge \frac{1}{3} |\arc{p_1p_3}|. 
\end{equation}
\end{itemize}

\end{prop}

\subsection{Two dimensional Alexandrov spaces}
If a complete length space $X$ has curvature bounded below by $K$, we say that $X$ is CBB($K$). When $X$ is two dimensional, $X$ is a topological surface with boundary $\partial X$ \cite[10.10.3]{BBI}. For any $p\in \partial X$, the space of direction $\Sigma_p X$ is isometric to $[0,\beta]$ for some $\beta \in (0,\pi]$. We say that $q\in \partial X$ is a corner with a turning angle $\theta$ if $\Sigma_q X$ is isometric to $[0,\pi-\theta]$ and $\theta \in (0,\pi)$. Otherwise, $q\in \partial X$ is called a boundary regular point. Note that $\theta$ can be calculated by 
\begin{equation}
    \pi-\theta = \lim_{p, m \to q} \angle_K \tilde p\tilde q\tilde m
\end{equation}
where the limit is taken along all $p, m\in \partial X$ converging to $q$ from opposite sides of $q$. Moreover, for $q\in \partial X$ one has the following angle equality \cite[p.127 Theorem 2]{AlexandrovselectedworksII}: Let $\sigma_1, \sigma_2, \sigma_2$ be three geodesics of $X$ starting at $q \in \partial X$ such that $\sigma_2$ lies in the middle of $\sigma_1, \sigma_3$: that is, for all $p'\in \sigma_1$, $m' \in \sigma_3$ close to $q$, there is a geodesic $\sigma$ joining $p', m'$ and interects $\sigma_2$. Then 
\begin{equation} \label{eqn angle equality}
    \angle (\sigma_1, \sigma_3) = \angle (\sigma_1, \sigma_2) + \angle (\sigma_2, \sigma_3). 
\end{equation}

\section{Representing convex curves in CBB($K$)} \label{section representing convex curves in CBB(K)}
Let $(X, d)$ be a two dimensional CBB($K$) space, which is homeomorphic to the closed unit disk. Let $\gamma$ be the boundary of $X$. We assume that $\gamma$ is rectifiable, has length $\ell>0$. In this section, we identify the boundary $\gamma=\partial X$ with a convex curve $\tilde \gamma$ in $S_K$, such that the (upper) arc/chord curvature of $\gamma$ agrees with the curvature of $\tilde\gamma$ almost everywhere. 

Let $V= (a_m)_{m=1}^\infty$ be a sequence of distinct points in $\gamma$, called the vertices of $V$. For each $m\ge 2$, let $\arc {I^m_1}, \cdots, \arc{I^m_m}$ be the connected components of $\gamma\setminus\{a_1, \cdots, a_m\}$. We assume that 
            \begin{equation} \label{eqn i_m to 0}
        i_m(V) := \max \{ |\arc {I^m_1}|, \cdots, |\arc{I^m_m}|\} \to 0 \ \ \text{ as } m\to \infty. 
    \end{equation}

The following construction will be crucial in future arguments. 

\begin{thm} \label{thm general construction of F}
   
    Let $V = (a_m)_{m=1}^\infty$ be a sequence of vertices such that (\ref{eqn i_m to 0}) holds. Then there is a convex curve $\tilde\gamma $ in $S_K$ and an arc-length preserving homeomorphism 
    $$f = f_V : \gamma \to \tilde\gamma,  \ \ s \mapsto \tilde s := f(s)$$ 
    with the following properties.
    \begin{itemize}
        \item [(i)] $f$ is distance non-increasing: 
        \begin{equation} |\tilde s_1\tilde s_2| \le |s_1s_2|_X, \ \ \forall s_1, s_2\in \gamma.
        \end{equation}
        \item [(ii)] $|a_1a_2|_X = |\tilde a_1 \tilde a_2|$ and if $a_{m}$ lies in the subarc $I^{m-1}_j = \arc{a_{i_1}a_{i_2}}$ for some $i_1, i_2 \in \{1, \cdots, m-1\}$, then 
        \begin{equation}
            |a_m a_{i_1}|_X = | \tilde a_m \tilde a_{i_1}|, \ \ |a_m a_{i_2}|_X = | \tilde a_m \tilde a_{i_2}|. 
        \end{equation}
        \end{itemize}
\end{thm}




\begin{proof}
    For each $k\ge 3$, let $P_k$ be the $k$-gon in $X$ with vertices $\{a_1, \cdots, a_k\}$. We construct by induction a sequence of convex $k$-gon $(\tilde P_k)_{k=3}^\infty$ with vertices $\{\tilde a_1, \cdots, \tilde a_k\}$ in $S_K$, such that (ii) in the statement of Theorem \ref{thm general construction of F} holds for any $3\le k \le m$. 

    Let $k=3$, let $\tilde P_3 = \triangle \tilde a_1 \tilde a_2 \tilde a_3$ be any comparison triangle in $S_K$ for $\triangle a_1 a_2 a_3$ in $X$. Clearly (ii) holds for $m=k=3$. 

    In general, assume that $\tilde P_N$ had been constructed so that (ii) holds for any $3\le m\le N$. Note that $a_{N+1}$ lies in the sub-arc $\arc{a_{i_1}a_{i_2}}$ for some $i_1, i_2 \in \{1, \cdots, N\}$. Let $\triangle \tilde a_{i_1} \tilde a_{N+1} \tilde a_{i_2}$ be the unique comparison triangle of $\triangle a_{i_1} a_{N+1} a_{i_2}$ that lies outside of the region bounded by $\tilde P_N$ and let $\tilde P_{N+1}$ be the $N+1$-gon in $S_K$ with vertices $\tilde a_1, \cdots, \tilde a_{N+1}$. Hence, the sequence $(\tilde a_m)_{m=1}^\infty$ and the sequence of $k$-gons $(P_k)_{k=3}^\infty$ in $S_K$ had been constructed by induction. 
    
    First, we show that $\tilde P_k$ is convex. For any vertex $\tilde a_j$ of $\tilde P_k$, let $\tilde a_i$, $\tilde a_k$ be its adjacent vertices. Then there are $\tilde a_{i_0}, \cdots, \tilde a_{i_n}$ such that $i_0 = i$, $i_n = k$, 
    \begin{equation} \label{eqn showing tilde P_k is convex}
    \angle_K \tilde a_i\tilde a_j \tilde a_k = \sum_{m=1}^n\angle_K \tilde a_{i_{m-1}} \tilde a_j\tilde a_{i_m},
    \end{equation}
    and each $\triangle \tilde a_{i_{m-1}}\tilde a_{j} \tilde a_{i_m}$, $m=1, \cdots, n$, is a comparison triangle used in constructing $\tilde P_k$. By the angle comparison $\angle_K \tilde a_{i_{m-1}} \tilde a_j\tilde a_{i_m} \le \angle a_{i_{m-1}} a_j a_m$, (\ref{eqn showing tilde P_k is convex}) and (\ref{eqn angle equality}), we conclude that $\angle_K \tilde a_i\tilde a_j \tilde a_k  \le \angle a_i a_j a_k \le \pi$ and hence $\tilde P_k$ is convex.

    For any $k$, let $f_k : P_k\to \tilde P_k$ be given by $f_k (a_i) = \tilde a_i$ for $i=1, \cdots, k$ and for any two adjacent vertices $a_i, a_j$, $f_k$ maps $[a_ia_j]$ isometrically onto $[\tilde{a}_i \tilde a_j]$. Next, we claim that $f_k$ is distance non-increasing: 
    \begin{equation} \label{eqn |tilde a_n tilde a_m| le |a_na_m|}
        |f_k(p)f_k(q)|\le |pq|_X, \ \ \text{ for all } p, q \in P_k
    \end{equation}
    We prove by induction. For $k=3$, this is more or less the definition of CBB($K$) space (e.g. \cite[4.3.8]{BBI}).


Now assume that the statement holds for $f_{m-1} : P_{m-1} \to \tilde P_{m-1}$. Note that $\tilde P_{m}$ is obtained by gluing $\triangle \tilde a_{i_1} \tilde a_m \tilde a_{i_2}$ to $\tilde P_{m-1}$ along the common edge $[\tilde a_{i_1}\tilde a_{i_2}]$.  Let $x, y\in P_m$. The chord $[a_{i_1}a_{i_2}]$ divides the curve $P_{m}$ into two parts: $P_{m-1}$ and $\triangle a_{i_1}a_ma_{i_2}$. If $x,y$ are in the same part, the conclusion follows from the induction hypothesis and the argument for $k=3$. Otherwise, if $x \in P_{m-1}$ and $y \in \triangle a_{i_1}a_m a_{i_2}$. Since $[a_{i_1}a_{i_2}]$ divides $X$ into two regions, $[xy]$ intersects $[a_{i_1}a_{i_2}]$ at a point $z$. By the induction hypothesis and the argument for $k=3$, 
\[\begin{split}
    |f_m (x)f_m(y)|&\leqslant |f_{m}(x)f_{m-1} (z)| +|f_{m-1}(z)f_{m}(y)|\\&\leqslant |xz|_X + |zy|_X \\
    &= |xy|_X.
\end{split}\]
Hence $f_m$ is also distance non-increasing, and this finishes the induction proof.

    Now we define $f=f_V$. For any $a_j$ define $f(a_j) = \tilde a_j$. Write $\tilde V \subset S_K$ be the collection of all $\tilde a_j$. Using (\ref{eqn |tilde a_n tilde a_m| le |a_na_m|}) and that $V$ is dense in $\gamma$, the map $f: V\to \tilde V$ extends uniquely to a distance non-increasing map $f : \gamma \to S_K$. Note that $\tilde \gamma := f(\gamma)$ is the limit of a sequence of convex $k$-gon $(P_k)_{k=1}^\infty$ and hence convex in $S_K$. 
    
    Let $a, b\in V$. Since $i_m(V)\to 0$ as $m\to \infty$, for any $n$ there are $n+1$ vertices $a^n_1, \cdots, a^n_n \in V\cap \arc{ab}$ so that $a^n_0 = a$, $a^n_n = b$ and $|a^n_ia^n_{i+1}|_X = |\widetilde{a^n_i}  \widetilde{a^n_{i+1}}|$ for all $i=0, \cdots, n-1$ and $\sum_i |a^n_ia^n_{i+1}|_X \to |\arc {ab}|$ as $n\to \infty$. Hence 
    $$|\arc{\tilde a\tilde b}|= \lim_n  \sum _i |\widetilde{a^n_i}\widetilde{a^n_{i+1}}| = \lim \sum_i |a^n_ia^n_{i+1}|_X =|\arc {ab}|$$
    Since $V$ is dense in $\gamma$, $|\arc{ab}|_X = |\arc{\tilde a\tilde b}|$ for any $a, b\in \gamma$. Hence $f$ is arc-length preserving. 
\end{proof}




We have the following corollary.
\begin{cor} \label{cor overline k (s)= k (tilde s) a.e. s}
Let $\gamma, \tilde\gamma, f$ be given in Theorem \ref{thm general construction of F}. Then 
\begin{equation} \label{eqn overline kappa_gamma le overline kappa_tilde gamma} 
\overline \kappa_{\gamma} (q) \le \overline{\kappa}_{\tilde\gamma}(\tilde q), \ \ \text{ for all }q\in \gamma. 
\end{equation}
and 
\begin{equation} \label{eqn overline kappa = kappa a.e.}
    \overline{\kappa}_{\gamma} (q) = \kappa_{\tilde\gamma} (\tilde q), \ \ \text{ for almost all } q\in \gamma.
\end{equation}
In particular, $\overline{\kappa}_{\gamma}(q)<\infty$ for almost all $q\in \gamma$.  
\end{cor}

\begin{proof}
Since $f$ preserves arc-lengths and is distance non-increasing, $\kappa_K (\arc{bc}) \le \kappa_K (\arc{\tilde b \tilde c})$ for all sub-arcs $\arc{bc}\subset \gamma$. Take $b, c\to q$ from both sides of $q \in \gamma$, we conclude (\ref{eqn overline kappa_gamma le overline kappa_tilde gamma}). $\tilde\gamma$ is a convex curve and therefore is twice differentiable on $\tilde\gamma^2 \subset \tilde\gamma$, where $\tilde\gamma$ has full measure. For any $q\in f^{-1}(\tilde\gamma^2)\setminus V$, there are sequences of vertices $(b_n), (c_n)$ in $V$ converging to $q$ from both sides with $|b_nc_n|_X = |\tilde b_n\tilde c_n|$. Thus
$$\kappa_K (\arc{b_nc_n}) = \kappa_K (\arc{\tilde b_n\tilde c_n})\to \kappa_{\tilde \gamma}(\tilde q).$$
Hence 
$\overline{\kappa}_{\tilde\gamma} (\tilde q) \ge \kappa_{\tilde\gamma} (\tilde q)$ for all $q\in f^{-1}(\tilde\gamma^2)\setminus V$. 
\end{proof}

\section{Globalization for the arc/chord curvature} \label{section globalization for the arc/chord}
In this section, we prove the globalization theorem for the arc/chord curvature (Theorem \ref{thm globalization for the arc/chord curvature}). We first show the equivalence of several curvature lower bound conditions for convex curves in $S_K$, where $K\ge 0$. The following proposition a is the first step in generalizing theorem  \ref{prop kappa_gamma ge g implies kappa_gamma ge g globally for curves in R^2} to convex curves in $S_K$ for $K>0$. 

\begin{prop} \label{prop kappa_gamma ge g implies underline kappa_gamma ge g globally for curves in S_K}
    Let $K>0$, and let $\sigma$ be a closed convex curve in $S_K$ such that $\kappa_\sigma(q) \ge \kappa$ almost everywhere. Then $\underline{\kappa}_\sigma(q)\ge \kappa$ for all $q\in \sigma$.  
\end{prop}

\begin{proof}
Fix $q\in\sigma$. By definition, there are sequence of points $(b_n)_{n=1}^\infty, (c_n)_{n=1}^\infty$ in $\sigma$ converging to $q$ such that $q\in \arc{b_nc_n}$ and 
\begin{equation} \label{eqn kappa(arc b_nc_n) to underline kappa_sigma (q)}  
\kappa_K(\arc{b_nc_n}) \to \underline{\kappa}_\sigma (q) \text{ as } n\to \infty.
\end{equation}
    For each $n$, let $V = (a_m)_{m=1}^\infty$ be a sequence of vertices in $\sigma$ such that $a_1 = b_n, a_2 = c_n$ and $i_m(V)\to 0$ as $m\to\infty$. Since $K>0$, $S_K$ is a CBB($0$) space. By Theorem \ref{thm general construction of F}, there is a convex curve $\tilde\sigma $ in $\mathbb R^2$ and an arc-length preserving, distance non-increasing map 
    $$f_1 : \sigma \to \tilde \sigma , \ \ q\mapsto \tilde q,$$
    such that $\kappa_\sigma(q) = \kappa_{\tilde\sigma} (\tilde q)$ for almost every $q\in \sigma$. Hence $\kappa_{\tilde \sigma}(\tilde q)\ge\kappa$ for almost all $\tilde q\in \tilde\sigma$. By theorem \ref{prop kappa_gamma ge g implies kappa_gamma ge g globally for curves in R^2}, $\kappa_0(\arc{\tilde b_n\tilde c_n}) \ge \kappa$. By the construction of $f_1$, $|b_nc_n|_X = |\tilde b_n\tilde c_n|$ and hence $\kappa _0 (\arc{b_nc_n}) = \kappa _0(\arc{\tilde b_n\tilde c_n}) \ge \kappa$. As a result,
   $$ \liminf_n \kappa_0 (\arc{b_nc_n}) \ge \kappa$$
and hence $\underline{\kappa}_\sigma(q) \ge \kappa$ by (\ref{eqn kappa(arc b_nc_n) to underline kappa_sigma (q)}) and remark \ref{rem under over kappa independent of K}.  
\end{proof}

We have the following globalization theorem for convex curves in $S_K$. 

\begin{thm} \label{thm globalization for convex curve in S_K}
Let $\sigma$ be a closed convex curve in $S_K$. Let $\kappa \in [0,\infty)$. The following are equivalent: 
\begin{itemize}
    \item [(a)] $\kappa_\sigma(q) \ge \kappa$ for almost all $q\in \sigma$, 
    \item [(b)] $\chi_\sigma(q) \ge \kappa$ for almost all $q\in \sigma$,
    \item [(c)] $\kappa_{\sigma, K}\ge \kappa$ globally, 
    \item [(d)] $\chi_{\sigma, K} \ge \kappa$ globally. 
\end{itemize}
\end{thm}

\begin{proof}
    (a) $\Leftrightarrow $ (b) since $\kappa_\sigma = \chi_\sigma$ almost everywhere. Clearly (c)$\Rightarrow$(a) and (d)$\Rightarrow$(b). On the other hand, it is shown that (d)$\Rightarrow$(c) holds for any metric space \cite[Remark 2.6]{AB}. To complete the proof, we show (a)$\Rightarrow$(d). By Proposition \ref{prop kappa_gamma ge g implies underline kappa_gamma ge g globally for curves in S_K}, it remains to show the following two claims. 
    
    \noindent {\bf Claim 1} If $\underline\kappa_\sigma(q) \ge \kappa$ for all $q\in\sigma$, then $\underline{\chi}_\sigma (q)\ge \kappa$ for all $q\in\sigma$. 
    
    {\sl Proof of claim  1}: Assume that $\underline{\chi}_\sigma(p)<\kappa$ for some $p\in \sigma$. Then by definition, there is $\kappa'<\kappa$ and two sequences of points $(\bar m_n)$, $(\bar q_n)$ in $\sigma$ converging to $p$ from both sides such that $\chi_K(\bar m_n, p, \bar q_n) \le \kappa'$. Since $|\bar m_n \bar q_n|\to 0$ as $n\to \infty$, for all large $n$, $p$ lies strictly inside $L(\bar m_n, \bar q_n)$, the $\kappa'$-lens with two corners $\bar m_n, \bar q_n$. Let $\sigma_n$ be the connected components of $\sigma\cap L(\bar m_n, \bar q_n)$ containing $p$, let $m_n, q_n$ be the endpoints of $\gamma_n$. Since $\gamma_n$ is convex and lies completely in $L(m_n, q_n)$, we have $\kappa_K(\gamma_n)\le \kappa'$. Taking $n\to\infty$, $\underline{\kappa}_\sigma (p) \le \kappa' <\kappa$ and this is a contradiction to the assumption.
    
    \noindent {\bf Claim 2}: If $\underline{\chi}_\sigma(q) \ge \kappa$ for all $q$, then $\chi_\sigma \ge \kappa$ globally. 
    
    {\sl Proof of claim  2}: Assume not. Then there are $p, q, r\in \sigma$ lying in a $\kappa'$ circle with $\kappa' < \kappa$. We may assume that the subarc $\sigma'$ of $\sigma$ joining $p$ to $q$ lies in the region bounded by the chord $pq$ and a $\kappa'$ minor arc. Now we translate the $\kappa'$-minor arc along the perpendicular bisector of $pq$ downward. Hence, one can find $q' \in \sigma'$ such that $\sigma'$ touches a $\kappa'$- minor arc at $q'$, and the $\kappa'$ minor arc lies within $\Omega$, the region bounded by $\sigma$. Hence $\underline\chi_\sigma(q') \le \kappa'<\kappa$, and this is a contradiction. 
\end{proof}

We are now ready to prove the globalization theorem for the arc/chord curvature. 

\begin{thm} \label{thm globalization for the arc/chord curvature}
    Let $X$ be a CBB($K$) space which is homeomorphic to the unit disk. $\gamma$ be its boundary. If $\overline{\kappa}_\gamma (s) \ge \kappa$ for almost all $s\in\gamma$. Then $\kappa_{\gamma, K} \ge \kappa$ globally. 
\end{thm}

\begin{proof}
    Let $\arc I = \arc{pq}$ be any sub-arc of $\gamma$ with endpoints $p, q$. Let $f: \gamma \to \tilde\gamma$ be the homeomorphism constructed in Theorem \ref{thm general construction of F} with $a_1 = p, a_2=q$. By (\ref{eqn overline kappa = kappa a.e.}), $\kappa_{\tilde\gamma}(s) \ge \kappa$ almost everywhere, and hence $\kappa_{\tilde\gamma} \ge \kappa$ globally by Theorem \ref{thm globalization for convex curve in S_K}. Hence $\kappa_K(\arc{pq}) = \kappa_K(\arc{\tilde p\tilde q}) \ge \kappa$ and $\kappa_\gamma \ge \kappa$ globally. 
\end{proof}

\section{The arc/chord curvature and the osculating curvature} \label{section kappa ge chi}
In this section, the prove the second main result, which an inequality between the arc/chord curvature and the osculating curvature of $\partial X$. This is proved by choosing a particular choice of vertices and applying Theorem \ref{thm general construction of F}. As a corollary we also prove the globalization for the osculating curvature. 

\begin{thm} \label{thm construction of F}
    Let $p, q, m$ be three points in $\gamma$. Then there is a closed convex curve in $\tilde \gamma$ in $S_K$ and an arc-length preserving, distance nonincreasing map
    $$ f: \gamma \to \tilde \gamma, \ \ s\mapsto \tilde s := f(s),$$
    such that 
        \begin{equation}
        |pq|_X = |\tilde p \tilde q|,\ \ |qm|_X = |\tilde q\tilde m|, \ \ |mp|_X = |\tilde m\tilde p|
        \end{equation}
        and
        \begin{equation} \label{eqn inequalities on chi, kappa}
            \overline\chi_{\gamma}(s) \ge \overline{\kappa}_{\gamma} (s)= \kappa_{\tilde \gamma}(\tilde s) \ge \underline{\chi}_\gamma ( s)
        \end{equation}
        for almost every $s\in \gamma$. 
\end{thm}



\begin{proof}[Proof of Theorem \ref{thm construction of F}]
We will construct a sequence of vertices $V_2= (a_m)_{m=1}^\infty$ in $\gamma$ and apply Theorem \ref{thm general construction of F}. 

\noindent (Step 1): Let $a_1 = p, a_2 = q, a_3 = m$.

\noindent (Step 2): The points $a_1, a_2, a_3$ divide $\gamma$ into three subarcs $\arc I^1_1, \arc I^1_2, \arc I^1_3$. For each $\arc I^1_i$, let $\arc J^1_i$ be the middle third of $\arc I^1_i$. That is, $\arc I^1_i\setminus \arc J^1_i $ consists of two connected components, each has arc length equals $|\arc I^1_i|/3$. Let $a_{i+3}$, $i=1, 2, 3$, be the point in $\arc J^1_i$ such that 
$$ \chi _K(b^1_{i,1}, a_{i+3}, b^1_{i, 2}) = \max_{q\in \arc J^1_i} \chi_K(b^1_{i,1}, q, b^1_{i,2}),$$
where $b^1_{i, 1}, b^1_{i, 2}\in \{a_1, a_2, a_3\}$ are the two endpoints of $\arc I^1_i$. 

\noindent (Step $k+1$) Let $g_k = 3\cdot 2^{k-1}$ and assume that $\{a_1, \cdots, a_{g_k}\} \subset \gamma$ had been found. Again, $\gamma\setminus\{a_1, \cdots, a_{g_k}\}$ consists of $g_k$ subarcs $\arc I^k_{j}$, $j=1, \cdots, g_k$. For each $j$, let $\arc J^k_j$ be the middle third of $\arc I^k_j$ and let $a_{j+g_k} \in \arc J^k_j$ such that 
$$ \chi_K (b^k_{j,1}, a_{j+g_k}, b^k_{j, 2}) = \max_{q\in \arc J^k_j} \chi_K(b^k_{j,1}, q, b^k_{j,2}),$$
where $b^k_{j, 1}, b^k_{j, 2}\in \{a_1,\cdots a_{g_k}\}$ are the two endpoints of $\arc I^k_j$. 

We have inductively constructed a sequence of vertices $V_2 = (a_m)_{m=1}^\infty$. Since 
$$|\arc I^{k+1}_j| \le \frac{2}{3} \max \{ |\arc I^k_1|, \cdots |\arc I^k_{g_k}|\},$$
one has $i_m(V_2) \to 0$ as $m\to \infty$. Hence Theorem \ref{thm general construction of F} is applicable and there is a closed convex curve $\tilde\gamma \subset S_K$ and an arc-length preserving, distance non-increasing map 
$$f_2 = f_{V_2}: \gamma \to \tilde\gamma, \ \ s\mapsto \tilde s$$
such that $\overline\kappa_\gamma(s) = \kappa_{\tilde\gamma}(\tilde s)$ for almost every $s\in \gamma$ by Corollary \ref{cor overline k (s)= k (tilde s) a.e. s}. It remains to show that the two inequalities in (\ref{eqn inequalities on chi, kappa}) hold for almost every $s\in \gamma$.

For each $k$, let $E_k \subset \gamma$ be the subset 
\begin{equation}
E_k = \bigcup_{i=1}^{g_k} \arc I^k_i \setminus \arc J^k_i.
\end{equation}

Define $\mathcal G$, $\mathcal B$ as 
\begin{align*}
\mathcal G = \{ q\in \gamma: \text{ there is a subsequence } (E_{k_j})_j \text{ such that } q\notin E_{k_j} \text{ for all }j\},
\end{align*}
and
\begin{align*}\mathcal B = \gamma \setminus \mathcal G = \bigcup_{n=1}^\infty \left( \bigcap_{k\ge n} E_k\right). 
\end{align*}

{\bf Claim}: $\mathcal B$ is of measure zero. 

{\sl Proof of claim }: It suffices to show that $\cap_{k\ge n}  E_k$ is of measure zero for all $n$. Now fix $n$. The statement is proved if we can find a subsequence $k_j$ such that $k_1 = n$ and  
    $$ \bigcap_{j=1}^\infty E_{k_j}$$
    is of measure zero. We will find inductively 
$$k=k_1<k_2 < \cdots < k_j <\cdots $$
such that for all $N$, 
    \begin{equation} \label{eqn induction on measure cap E_k_j}
        \left| \bigcap_{j=1}^N E_{k_j}\right|< \left(\frac{3}{4}\right)^N \ell. 
    \end{equation}
    The $N=1$ case is trivial since $| E_{k}| = \frac 23 \ell <\frac 34 \ell$. Assume that (\ref{eqn induction on measure cap E_k_j}) holds for some $N$. Note that $ E_{k_1}\cap \cdots \cap E_{k_N}$ consists of $M=M_N$ disjoint subarcs $\arc I_1, \cdots, \arc I_M$ of $\gamma$. Let $\epsilon>0$ be small so that 
$$\frac 23 |\arc I_\alpha|+2\epsilon < \frac 34 |\arc I_\alpha|, \ \ \alpha=1, \cdots, M.$$
Let $k_{N+1} >k_{N}$ such that if $m = k_{N+1}-1$, then each subarc $\arc I^m_j$ has length less than $\epsilon$ for any $i=1, \cdots, g_m$. 

Fix $\arc I_\alpha$. Let $\arc I$ be any subarc of the form $\arc I^m_j$ lying in $\arc I_\alpha$. Since $E_{k_{N+1}} \cap \arc I^m_j = \arc I^m_j\setminus \arc J^m_j$, we see that 
$$| E_{k_{N+1}}\cap \arc I_\alpha| \le \frac 23 |\arc I_\alpha | + 2 \epsilon< \frac 34 |\arc I_\alpha|$$
Summing over $\alpha=1, \cdots, M$ and using the induction hypothesis, we conclude 
\begin{align*}
    \left| \bigcap _{j=1}^{N+1} E_{k_j} \right| = \sum_{\alpha=1}^M \left| E_{k_{N+1}} \cap I_\alpha \right| < \frac 34 \sum_{\alpha=1}^M |I_\alpha| = \frac 34 \left|\bigcap _{j=1}^{N} E_{k_j} \right| < \left( \frac 34\right)^{N+1} \ell
\end{align*}
and hence we have finished the induction proof. 

Now we return to the proof of Theorem \ref{thm construction of F}. 

First, By construction of $F_2$, for any $q\in \gamma\setminus V_2$, there are sequences of vertices $(b_n)_{n=1}^\infty$, $(c_n)_{n=1}^\infty$ converging to $q$ from both sides of $q$ such that $|b_nc_n|_X = |\tilde b_n \tilde c_n|$ for all $n$. Since $F_2$ is distance non-increasing, $|b_n q|_X\ge |\tilde b_n \tilde q |$ and $|qc_n|_X \ge |\tilde q \tilde c_n|$. This implies 
\begin{equation}
 \chi _K(b_n, q, c_n)\ge \chi_K(\tilde b_n, \tilde q, \tilde c_n) 
\end{equation}
for all $n$. Hence  
\begin{equation}
 \overline{\chi} _\gamma(q) \ge \limsup_{m\to \infty}  \chi _K(a_m, q, b_m)\ge \lim_{m\to \infty} \chi _K(\tilde a_m, \tilde q, \tilde b_m) =\kappa_{\tilde\gamma}(\tilde q). 
\end{equation}
whenever $\tilde \gamma$ is twice differentiable at $\tilde q$ and $q\notin V_2$. Hence $\overline{\chi}_\gamma(q) \ge \kappa_{\tilde\gamma} (\tilde q)$ for almost all $s\in \gamma$. 

It remains to show that 
\begin{equation} \label{eqn kappa_{tilde gamma} ge underline{chi}_gamma}
\kappa_{\tilde \gamma} (\tilde q) \ge \underline{\chi}_\gamma(q) \ \ \text{ for almost all } q\in\gamma. 
\end{equation}
First, let $\tilde V_2 = F(V_2) = (\tilde a_m)_{m=1}^\infty$ be the sequence of vertices in $\tilde\gamma$. Since $F$ preserves arc-length, $i_m (\tilde V_2)\to 0$ as $m\to \infty$. Hence by Theorem \ref{thm general construction of F}, there is a convex curve $\bar \gamma \subset \mathbb R^2$ and a homeomorphism $\tilde F:\tilde\gamma \to \bar \gamma$, $\tilde s \mapsto \bar s$ such that $\kappa_{\tilde \gamma} (\tilde s) = \kappa_{\bar \gamma}(\bar s)$ for almost all $\tilde s\in\tilde \gamma$. As a result, to prove (\ref{eqn kappa_{tilde gamma} ge underline{chi}_gamma}) we might as well assume that $\tilde\gamma \subset \mathbb R^2$ by identifying $\tilde \gamma$ as $\bar \gamma$,

Let $ q \in \mathcal G$. Then there is a sequence $(k_j)_{j=1}^\infty$ such that $q\notin E_{k_j}$ for all $j$. Hence, for each $j$ there is $i \in\{ 1, \cdots, g_{k_j}\}$ such that $q\in J^{k_j}_i$. By the choice of $a_{i+g_{k_j}}$, we have 
\begin{equation} \label{eqn chi (A s A) le chi (A q A)}
\chi _K (b^{k_j}_{i,1}, q, b^{k_j}_{i,2}) \le \chi_K (b^{k_j}_{i,1}, a_{i+g_{k_j}}, b^{k_j}_{i,2})=\chi _K(\tilde b^{k_j}_{i,1}, \tilde a_{i+g_{k_j}}, \tilde b^{k_j}_{i,2}).
\end{equation}
The sequence of comparison triangles $(\widetilde\triangle _j)_{j=1}^\infty$, where 
$$\widetilde \triangle_j = \triangle \tilde b^{k_j}_{i,1} \tilde a_{i+g_{k_j}} \tilde b^{k_j}_{i,2}$$ 
in $\mathbb R^2$ converges to $\tilde q$ as $j\to\infty$. Since $\tilde a_{i+g_{k_j}} \in \arc{J}^{k_j}_i$, $s\in \arc{\tilde b^{k_j}_{i,1}\tilde b^{k_j}_{i,2}}$
 and 
 $$|\arc{\tilde b^{k_j}_{i,1} \tilde a_{i+g_{k_j}}}| , |\arc{\tilde a_{i+ g_{k_j}} \tilde b^{k_j}_{i,2}}| \ge \frac{1}{3} |\arc{\tilde b^{k_j}_{i,1} \tilde b^{k_j}_{i,2}}|$$

Hence by Proposition \ref{prop calculate chi by not so thin triangles},
$$ \kappa_{\tilde\gamma} (\tilde q) = \lim_{j\to\infty} \chi _K(\widetilde \triangle _j)$$
whenever $\tilde \gamma$ is twice differentiable at $\tilde q$. Together with (\ref{eqn chi (A s A) le chi (A q A)}) one concludes 
\begin{equation}
    \underline \chi_\gamma(q) \le \kappa_{\tilde \gamma}(\tilde s) 
\end{equation}
whenever $\tilde\gamma$ is twice differentiable at $\tilde q$ and $q\in \mathcal G$. Hence it holds for almost all $q\in \gamma$. 
\end{proof}

\begin{thm}[Globalization theorem for the osculating curvature] \label{thm globalization for chi in X}
    Let $\gamma$ be a convex curve in a CBB($K$) space. Then $\underline{\chi}_{\gamma} (q) \ge \kappa$ for almost every $q\in \gamma$ if and only if $\chi_{\gamma, K} \ge \kappa$ globally. 
\end{thm}

\begin{proof}
    ($\Leftarrow$) is obvious, and we prove only $(\Rightarrow)$. Let $p, q, m \in \gamma$ be any three distinct points. Let $F : \gamma \to \tilde\gamma$ be the map constructed in Theorem \ref{thm construction of F}. Then $\chi _K(p, q, m) = \chi_K(\tilde p, \tilde q, \tilde m)$. On the other hand, by (\ref{eqn inequalities on chi, kappa}), $\kappa_{\tilde\gamma}(\tilde q)\ge \kappa$ for almost every $\tilde q\in \tilde\gamma$. By Theorem \ref{thm globalization for convex curve in S_K}, $\chi_{\tilde\gamma, K} \ge \kappa$ globally and hence $\chi_K(\tilde p, \tilde q, \tilde m) \ge \kappa$. Hence $\chi_K (p, q, m)\ge \kappa$ and this finishes the proof of the theorem. 
\end{proof}

\section{Rigidity of the $\kappa$-lens} \label{section rigidity}
The main theorem of this section is the following rigidity theorem. 

\begin{thm} \label{thm length bounds of gamma with arc/chord lower bound and rigidity}
Let $\gamma$ be a convex curve in $X$ such that $\overline{\kappa}_\gamma(q) \ge \kappa$ for almost every $q\in \gamma$. Assume that $\gamma$ has a corner at $q$ with turning angle $\theta \in [0, \pi)$. Then the arc-length of $\gamma$ satisfies 
\begin{equation} \label{eqn length bounds}
    L(\gamma) \le L(\partial L_{\kappa, \theta}),
\end{equation}
and equality holds if and only if $X$ is isometric to a $\kappa$-lens in $S_K$ with turning angle $\theta$.
\end{thm}

\begin{proof}[Proof of Theorem \ref{thm length bounds of gamma with arc/chord lower bound and rigidity}] 
Let $V$ be a sequence of vertices on $\gamma$ such that $q \in V$ and $i_m (V)\to 0$ as $m\to \infty$. Let $f=f_V : \gamma\to \tilde \gamma$ be the homeomorphism constructed in theorem \ref{thm general construction of F}. By Corollary \ref{cor overline k (s)= k (tilde s) a.e. s}, $\kappa_{\tilde\gamma} (\tilde s) = \overline{\kappa}_\gamma(s)$ for almost every $s\in \gamma$. Hence $\kappa_{\tilde\gamma} (\tilde s) \ge \kappa$ for almost all $\tilde s\in \tilde\gamma$.

Since $q$ is one of the vertices, one can find two sequences  $(c_j)_{j=1}^\infty$ $(d_j)_{j=1}^\infty$ of vertices of $V$ that converge to $q$ from different sides of $q$ and 
\begin{equation} \label{eqn sequence of angles summation 1}
\angle c_1qd_1 + \sum_{j=1}^\infty \angle c_j q c_{j+1} + \sum_{j=1}^\infty \angle d_j q d_{j+1} = \pi - \theta.
\end{equation}
Since $X$ is CBB$(K)$, the angle comparison implies 
\begin{equation} \label{eqn sequence of angles summation 2}
\angle _K \tilde c_1 \tilde q \tilde d_1 + \sum_{j=1}^\infty \angle_K  \tilde c_j \tilde q \tilde c_{j+1} + \sum_{j=1}^\infty \angle_K \tilde d_j \tilde q \tilde d_{j+1} \le \pi - \theta,
\end{equation}
that is, $\tilde \gamma$ has a corner at $\tilde q$ with turning angle $\tilde\theta \ge \theta$. Hence $\tilde q$ admits two supporting lines which intersect at an angle $\tilde \theta$. By Lemma \ref{lem kappa_sigma ge kappa_0 implies sigma lies inside a kappa_0 circle}, $\tilde\gamma$ lies inside a $\kappa$-lens $\mathcal L_{\kappa, \theta}$ in $S_K$ with turning angle $\tilde \theta$. Together with the convexity of $\tilde\gamma$, 
$$L(\gamma) = L(\tilde\gamma) \le  L( \partial \mathcal{L}_{\kappa, \tilde\theta}) \le L(\partial \mathcal {L}_{\kappa, \theta}).$$

In the rest of the proof, we assume that equality in (\ref{eqn length bounds}) holds and prove that $X$ is isometric to the $\kappa$-lens $\mathcal L_{\kappa, \theta}$. 
The case $\kappa = 0$ has been proved in \cite{GP}. From now on we assume $\kappa >0$. In this case, the turning angle $\tilde\theta$ at $\tilde q$ equals to $\theta$ and $\tilde\gamma = \partial \mathcal L_{\kappa, \theta}$. 

\noindent {\bf Claim 1}: one has 
\begin{equation}
    |xy|_X = |\tilde x\tilde y|, \ \ \text{ for all } x ,y \in \gamma.
\end{equation} 
{\sl Proof of claim  1}: Let $x, y \in \gamma$, let $\bar V = (\bar a_m)_{m=1}^\infty$ be a sequence of vertices in $\gamma$ such that $\{ q, x, y\} \subset \{\bar a_1, \bar a_2, \bar a_3\}$ and $i_m (\bar V) \to 0$ as $m\to \infty$. Let $\bar f =f_{\bar V}:\: \gamma \to \bar \gamma$, $s\mapsto \bar s$ be the arc-length preserving homeomorphism constructed in Theorem \ref{thm general construction of F}. Then $\bar \gamma$ is also the boundary of a $\kappa$-lens in $S_K$ with turning angle $\theta$. Composing with a rigid motion of $S_K$ if necessary, we may assume that $\tilde\gamma = \bar\gamma$ and $\tilde q = \bar q$. Since $f,\bar f$ both are arc-length preserving, one must have $f= \bar f$ or $f = R\bar f$, where $R$ is a reflection on $S_K$. As a result, $|xy|_X = |\bar x\bar y| = |\tilde x \tilde y|$.

Let $r \in \gamma$ such that $\tilde r$ is the other corner of $\tilde\gamma = \partial \mathcal L_{\kappa, \theta}$. This is the unique point in $\gamma$ such that $\gamma \setminus \{ q, r\}$ splits into two sub-arcs of equal lengths. 

\noindent {\bf Claim 2}: $\gamma$ also has a corner at $r$ with turning angle $\theta$. 

{\sl Proof of claim  2}: This follows from the fact that, by claim 1, there is an isometry $R : \gamma \to \gamma$ such that $R(q) = r$ and $|s_1s_2|_X = |R(s_1)R(s_2)|_X$ for all $s_1, s_2\in \gamma$. 

\noindent {\bf Claim 3}: For every three points $a, b, c \in \gamma$ and two given geodesics $[ba], [bc]$, we have $\angle ([ba], [bc]) = \angle_K \tilde a\tilde b\tilde c$. 

{\sl Proof of claim  3}: By the angle comparison we always have 
$$\angle ([ba], [bc]) \ge \angle_K \tilde a\tilde b\tilde c.$$ 
When $b =q$, choose $c_1=a, d_1 = b$ in (\ref{eqn sequence of angles summation 1}). Together with (\ref{eqn sequence of angles summation 2}) and that $\theta = \tilde\theta$, we conclude $\angle abc = \angle_K \tilde a\tilde b\tilde c$. Similar when $b=r$.

When $b\notin \{ q, r\}$, $\tilde b$ is a smooth point of $\partial \mathcal L_{\kappa, \theta}$. The turning angle at $\tilde b$ (and hence $b$) is zero. By picking a sequence of vertices $\{c_j\}, \{d_j\}$ converging to $b$ from both sides of $b$ with $c_1 = a$, $d_1 = c$, we conclude that $\angle abc = \angle_K \tilde a\tilde b\tilde c$ as in the case for $b=q$. 

\noindent {\bf Claim 4}: For every two $a, b\in \gamma$, there exists a unique geodesic in $X$ joining $a$ to $b$. 

{\sl Proof of claim  4}: Such a geodesic exists since $X$ is CBB($K$). To show uniqueness, let $\sigma_1, \sigma_2$ be two geodesics in $X$ joining $a$ to $b$. Let $c$ be any point in $\gamma \setminus \{a, b\}$, and let $\sigma$ be a geodesic joining to $a$ to $c$. By claim 3,  $\angle (\sigma_1, \sigma) = \angle_K \tilde b\tilde a\tilde c = \angle (\sigma_2, \sigma)$ and the angle between $\sigma_1$ and $\sigma_2$ is zero. Hence $\sigma_1 = \sigma_2$. 

\noindent {\bf Claim 5}: For each $x\in X\setminus \partial X$ and $p\in \partial X$, there exists a unique $p_x\in \partial X$ such that the geodesic $\sigma$ joining $p$, $p_x$ passes through $x$. 

{\sl Proof of claim  5}: Since $x$ lies in the interior, such a geodesic must be unique. To show the existence, let $\gamma : [0, \ell]\to \partial X$ a parametrization of $\partial X$ such that $\gamma(0) = p$. For each $t\in [0,\ell]$, let $\sigma_t $ be the unique geodesic joining $p$ to $\gamma(t)$. The geodesic $\sigma_t$ divides $X$ into two connected components $X^+_t, X^-_t$, where $X^+_t$ is the component such that $\gamma (\bar t) \in X^+_t$ for all $\bar t >0$, $\bar t$ small. 

Let $I \subset (0,\ell)$ such that $t\in I$ if and only if $x\in X^-_t$. Note that $t\in I$ when $t$ is small. Also, if $\bar t\in I$, then $(0,\bar t) \subset I$. Similarly, let $J\subset (0,\ell)$ such that $t\in J$ if and only if $x\in X^+_t$. Then $\bar t \in J$ implies $(\bar t, \ell )\subset J$. By definition, $I\cap J = \emptyset$. 

We argue by contradiction that $I\cup J \neq (0,\ell)$ as follows: if $I\cup J = (0,\ell)$, then $I$, $J$ share a common endpoint $t_0 \in (0,\ell)$. 

If $t_0 \notin I$, let $(t_j)_{j=1}^\infty$ be an increasing sequence in $I$ that converges to $t_0$. Taking a subsequence of $\{t_j\}$ if necessary, we may assume that $(\sigma_{t_j})_{j=1}^\infty$ converges uniformly to a geodesic connecting $p$ and $\gamma(t_0)$. By Claim 4 this geodesic must be $\sigma_{t_0}$. Since $x\in X^-_{t_j}$ for all $j$, we have $x\notin X^+_{t_0}$ and hence $t_0 \notin J$. This is not possible since $I\cup J$ contains $t_0$. As a result, $t_0 \in I$. However, the same argument shows that $t_0 \notin J$ is impossible. Hence, $t_0\in I\cap J$, which is impossible by the definition of $I, J$. 

Thus, there is $t_0 \in (0,\ell)$ not in $I\cup J$ and the geodesic $\sigma_{t_0}$ contains $x$. This finishes the Proof of claim  5. 

\noindent{\bf Claim 6}: For any $x, y\in X\setminus \partial X$, there is a unique geodesic in $X$ joining $x$, $y$ and two points on $\partial X$. 

{\sl Proof of claim  6}: Let $x ,y \in X\setminus \partial X$ be given. Let $\sigma$ be any geodesic constructed in claim 5 that passes through $x$. We are done if $\sigma$ contains $y$. If not, let $Y$ be the connected component of $X\setminus \sigma$ containing $y$. Let $\gamma : [0, \ell] \to X$ be a parametrization of the boundary, and let $t_1<t_2$ such that $\gamma(t) \in Y$ if and only if $t\in (t_1, t_2)$. Now for each $t\in (t_1, t_2)$, let $\sigma_t$ be the geodesic constructed in claim 5, which passes through $x$ and $\gamma(t)$. Using the same continuity argument as in claim 5, one can find $t_3 \in (t_1, t_2)$ such that the geodesic $\sigma_{t_3}$ passes through $y$. 

Now we are ready to extend the homeomorphism $f : \gamma \to \tilde \gamma$ to an isometry $F : X\to \mathcal L_{\kappa, \theta}$. Let $p\in \partial X$ be fixed. For any $x\in X$. Let $\sigma$ be the unique geodesic in $X$ which contains $x$ and joining $p$ to some $s\in \gamma$. Let $[\tilde p\tilde s]$ be the unique geodesic in $\mathcal L_{\kappa, \theta}$ joining $\tilde p$, $\tilde s$. We define $F(x)=\tilde x$ to be the unique point on $[\tilde p\tilde s]$ such that $|px|_X = |\tilde p \tilde x|$. 

\noindent {\bf Claim 7}: The map $F :X\to \mathcal L_{\kappa, \theta}$ is independent of $p \in \partial X$ chosen. 

{\sl Proof of claim  7}: Let $p_1 \in \partial X$, and let $[p_1s_1]$ be the geodesic passing through $p_1$, $x$, and $s_1\in \partial X$. It suffices to show that $\tilde x \in [\tilde p_1\tilde s_1]$ and $|p_1x|_X = |\tilde p_1\tilde x|$. By claim 3, one has $\angle ([pp_1], [ps]) = \angle_K \tilde p_1\tilde p\tilde s$, which implies $|px|_X = |\tilde p\tilde x|$. Similarly, one has $|s_1x|_X = |\tilde s_1\tilde x|$. Hence 
$$ |\tilde p_1 \tilde s_1|= |p_1s_1|_X = |p_1x|_X+ |xs_1|_X = |\tilde p_1 \tilde x| + |\tilde x\tilde s_1|$$
and $\tilde x \in [\tilde p_1\tilde s_1]$. Thus claim 7 is proved.

Lastly, we finish the proof of the theorem: let $x, y\in X$. Let $\sigma$ be the geodesic in $X$ containing $x, y$ and joining $p, s\in \partial X$. We may assume that $|py|_X \ge |px|_X$. By claim 7,
\begin{align*}
    |\tilde x\tilde y| = |\tilde p\tilde y |- |\tilde p \tilde x| = |py|_X - |px|_X = |xy|_X.
\end{align*}
Hence $F : X\to \mathcal L_{\kappa, \theta}$ is an isometry. 
\end{proof}

The rigidity theorem below follows directly from Theorem \ref{thm length bounds of gamma with arc/chord lower bound and rigidity} and (\ref{eqn inequalities on chi, kappa}). 

\begin{thm} \label{thm length bounds of gamma with chi lower bound and rigidity}
    Let $\gamma = \partial X$ be the boundary of a two dimensional CBB($K$) $X$ that is homeomorphic to the closed unit disk. Assume that $\underline\chi_{\gamma}(q)\ge \kappa$ for almost every $q\in\gamma$ and that $\gamma$ has a corner at $q \in \gamma$ with turning angle $\theta$. Then 
    $$ L(\gamma) \le L(\mathcal L_{\kappa, \theta})$$
    and equality holds if and only if $X$ is isometric to the $\kappa$-lens with turning angle $\theta$ in $S_K$. 
\end{thm}

\section{Appendix A: Geodesic curvatures for convex curves in Euclidean plane}
In this appendix, we show that the arc/chord curvature and the osculating curvature of a convex curve $\sigma$ in $\mathbb R^2$ exists and agrees almost everywhere. If $\sigma$ is locally the graph of a function $f$, the curvature agrees with the second dervatives of $f$; if $\sigma$ is given an arc-length parametrization, then the curvature at $\gamma(s)$ is $\theta'(s)$, where $\theta$ is the angle that $\sigma'(s)$ makes with the $x$-axis.

First we consider the graphical representation. Let $\sigma$ be a closed convex curve in $\mathbb R^2$ and $q\in \sigma$. Translating and rotating $\sigma$ if necessary, we may assume that $q = (0,0)$ and locally at $(0,0)$, $\sigma$ is the graph of a convex function $f$ with $f(0) = 0$. By the convexity of $\sigma$, 
$$f'_\pm(t):= \lim_{s\to t^\pm} \frac{f(s) - f(t)}{s-t}$$ 
exists for all $t$, are nondecreasing with respect to $t$ and $f'_+(s)\le f'_-(t)\le f'_+(t^+)$ for all $s<t$. Hence $f'_\pm(t)$ are continuous and differentiable almost everywhere. From now on we assume that $f'(0)$ exists and that both $f'_\pm(t)$ are differentiable at $t=0$ with derivative $A\ge 0$. Hence

\begin{equation} \label{eqn f = At + o(t)}
f'_\pm (t) = At + o(t)    \ \ \text{ as } t\to 0,
\end{equation}
and
\begin{equation} \label{eqn f = Ax^2 + o(x^2)}
f(t) = \frac{A}{2}t^2 + o(t^2), \ \ \text{ as } t\to 0.
\end{equation}
This is just the Alexandrov Theorem in one dimension. 
 
 Let $p_1, p_2, p_3$ be three points on $\sigma$ of the form $p_i = (t_i, f(t_i))$ for $i=1, 2, 3$. Write $t_{\text{max}} = \max\{|t_1|, |t_2|,|t_3|\}$. By direct calculation, when $i\neq j$, 
$$ |p_ip_j|^2 = |t_i-t_j|^2 \left( 1+ \frac{A^2}{4} (t_i+t_j)^2 + o(t_{\text{max}}^2)\right)$$
and hence 
\begin{align} \label{eqn calculate of product of 3 side lengths}
    |p_1p_2|\cdot &|p_2p_3| \cdot |p_3p_1|\\
    \notag &= |t_1-t_2|\cdot | t_2 - t_3| \cdot |t_3 - t_1| (1+ O(t_{\text{max}})). 
\end{align}
On the other hand, let $\mathcal A$ be the area of the triangle $\triangle p_1p_2p_3$. Then 
\begin{align} \label{eqn A = 1/2 det}
    \mathcal A = \frac{1}{2} |\det (\vec{p_1p_2}, \vec {p_3p_2})| 
    \end{align}
and the osculating curvature is given by 
\begin{align} \label{eqn dfn osculating curvature in R^2}
    \chi_0 (p_1,p_2,p_3) &= \frac{4\mathcal A}{|p_1p_2|\cdot |p_2p_3| \cdot |p_3p_1|}.
\end{align}

\begin{prop} \label{prop osculating curvature =A}
The osculating curvature of $\sigma$ at $q$ equals $A$. 
\end{prop}
\begin{proof}
We assume $t_2 = 0$ and $t_1 < 0<t_3$ in the above calculations. Using (\ref{eqn f = Ax^2 + o(x^2)}) and (\ref{eqn A = 1/2 det}), we have 
\begin{align*}
    \mathcal A &= \frac{1}{ 2}\left| \frac{A}{2} t_1t_3^2 -\frac{A}{2} t_3 t_1^2 + t_1 o (t_3^2) - t_3 o(t_1^2)\right| \\
    &= \frac{A}{4} |t_1t_3| (|t_3 -t_1| + o(t_{\text{max}}))
\end{align*}
Since $t_1<0<t_3$, $|t_3-t_1|\ge  t_{\text{max}}$. Hence 
\begin{equation}
\frac{o(t_{\text{max}})}{|t_3-t_1|} \to 0, \ \ \text{ as } t_1, t_3 \to 0.
\end{equation} 
Together with (\ref{eqn dfn osculating curvature in R^2}), (\ref{eqn calculate of product of 3 side lengths}), we conclude that $\chi_\sigma (q)=A$.
\end{proof}

Next, we prove Proposition \ref{prop calculate chi by not so thin triangles}.
\begin{proof}[Proof of Proposition \ref{prop calculate chi by not so thin triangles}]
In general, assume that $p_1, p_2, p_3$ satisfy (\ref{eqn |arc p_1p_2| ge 1/3 |arc p_1p_3|}). Using
\begin{equation*}
    f'_+(t_1) \le \frac{f(t) - f(s)}{t-s} \le f'_-(t_2), 
\end{equation*}
for each $t_1< s<t<t_2$, we have $|f(t)-f(s)|\le \max\{|f'_+(t_1)|, |f'_-(t_2)|\} |t-s|$. By (\ref{eqn f = At + o(t)}), $f'_\pm(t)\to 0$ as $t\to 0$. Hence 
\begin{equation}
    |t_2-t_1|\le |\arc{p_1p_2}|\le |t_2-t_1| (1+ C(t_1, t_2)),
\end{equation}
where $C(t_1, t_2) \to 0$ as $t_1, t_2\to (0,0)$. Hence, if $p_1, p_2, p_3$ are closed enough to $q$, 
$$ |t_1-t_2|, |t_2- t_3| \ge \frac 14 |t_3-t_1|. $$
Note also that $t_1<0<t_3$, hence $|t_3-t_1|\ge t_{\text{max}}$. Then we have 
$$ |t_2-t_1|\cdot |t_3-t_2| \cdot |t_1-t_3| \ge \frac{1}{16} t^3_{\text{max}}.$$
This implies 
$$ \frac{o(t^3_{\text{max}})}{|t_2-t_1|\cdot |t_3-t_2| \cdot |t_1-t_3|}\to 0, \ \ \text{ as } t_{\text{max}}\to 0.$$
Lastly, note that by (\ref{eqn A = 1/2 det}),
\begin{equation}
    \mathcal A = \frac{1}{4} |A(t_2-t_1)(t_3-t_2)(t_1-t_3) + o(t_{\text{max}}^3)|.
\end{equation}
Together with (\ref{eqn dfn osculating curvature in R^2}) and (\ref{eqn calculate of product of 3 side lengths}), 
\begin{equation}
    A =\lim_{p_1, p_2, p_3} \chi_0 (p_1, p_2, p_3),
\end{equation}
where the limit is taken along all $p_1, p_2, p_3\in \sigma$ which satisfy (i)-(iii) in Proposition \ref{prop calculate chi by not so thin triangles}. This finishes the proof of the Proposition. 
\end{proof}

Next, we show that the arc/chord curvature of $\sigma$ exists almost everywhere and is equal to the osculating curvature. Using the same graphical representation of $\sigma$ by $f$ and Proposition \ref{prop osculating curvature =A}, it suffices to show

\begin{prop} \label{prop arc/chord =A}
    The arc/chord curvature of $\sigma$ at $q$ equals $A$. 
\end{prop}
\begin{proof}
For any $a<0<b$, let 
\begin{align*}
    \ell &= \ell(a, b) = \int_a^b \sqrt{1 + (f')^2}, \\
    r &= r(a, b) = \sqrt{(b-a)^2 + (f(b)-f(a))^2}. 
\end{align*}
By remark \ref{rem under over kappa independent of K} and that $\lim_{a, b\to 0} r/(b-a) = 1$, it suffices to show that 
\begin{equation} \label{eqn limit of l-r /(b-a)^3 = A^2/24}
\lim_{a, b\to 0} \frac{\ell -r}{(b-a)^3} = \frac{A^2}{24}.
\end{equation}
Using the estimates $\sqrt{1+x} = 1 + \frac{x}{2} + O(x^2)$, (\ref{eqn f = Ax^2 + o(x^2)}) and note that $b-a \ge \max\{|a|, b\}$, 
\begin{equation}
    \ell = b-a + \frac{A^2}{6} (b^3-a^3) + O((b-a)^5). 
\end{equation}
Similarly, using 
$$ \frac{f(b) - f(a)}{b-a} \le \max\{|f'_+(a)|, f'_-(b^-)\},$$
and (\ref{eqn f = Ax^2 + o(x^2)}),
\begin{align*}
r &= \sqrt{(b-a) ^2 + (f(b)-f(a))^2} \\
&= b-a + \frac{1}{2} \frac{(f(b) - f(a))^2}{b-a} + O \left((b-a)^4\right). 
\end{align*}
For the second term on the right, we use (\ref{eqn f = Ax^2 + o(x^2)}) again: 
\begin{align*}
    \frac 12 \frac{(f(b) - f(a))^2}{b-a} &= \frac{A^2}{8} (b-a) (b+a)^2 + o((b-a)^3)
\end{align*}
Hence 
\begin{equation}
\ell-r = \frac{A^2}{24} (b-a)^3 + O((b-a)^4),
\end{equation}
which implies (\ref{eqn limit of l-r /(b-a)^3 = A^2/24}). 
\end{proof}

Lastly, when $s$ is the arc-length parametrization of $\sigma$ and $\theta(s)$ is the angle $\sigma'(s)$ makes with the $x$-axis, the following can be verified by a direct calculation. 

\begin{prop}
    For almost all $s$, $\theta'(s)$ exists and agrees with the arc/chord curvature. 
\end{prop}

\section{Appendix B: base angle/chord curvature}
\begin{dfn}
Let $X$ be a $2$-dimensional Alexandrov space with curvature $\ge K$ and $\partial X \neq \emptyset$. Let $p \in \partial X$ and let $\gamma$ be a parametrization of $\partial X$. The lower (resp. upper) base angle/chord curvature of $\partial X$ at $p$ is 
\begin{equation}
    \underline\lambda _\gamma (p) = \liminf \frac{2\alpha}{|pq|_X} \ \ \left( \text{ resp. } \overline\lambda _\gamma (p) = \limsup \frac{2\alpha}{|pq|_X} \right), 
\end{equation}
where the limit is taken over all $q\in \partial X$ converging to $p$ and $\alpha$ is the angle between $[pq]$ and the arc $\arc{pq}\subset \gamma$.
\end{dfn}

\begin{prop}
For all $q \in \partial X$, 
\begin{equation} \label{eqn unerline lambda le underline chi}
    \underline{\lambda}_{\gamma }(q) \le \underline\chi_{\gamma} (q). 
\end{equation}
\end{prop}

\begin{proof}
    Write $\lambda = \underline{\lambda}_{\gamma} (q)$. We assume that $\lambda<\infty$. The case $\lambda = \infty$ can be proved similarly. Let $\epsilon >0$. Then one has 
    \begin{equation}
        \frac{2\alpha_p}{|pq|_X} \ge \lambda-\epsilon, \ \ \frac{2\alpha_m}{|qm|_X} \ge \lambda-\epsilon
    \end{equation}
    whenever $|pq|_X,|qm|_X$ are small, where $\alpha_p = \angle ([qp], \arc{qp})$ and similar for $\alpha_m$. Assume that $p, m$ lies on different sides of $q$. Then  
    \begin{align*}
        \chi_0(p,q,m) &= \frac{2\sin \angle_0 pqm }{|pm|_X}
    \end{align*}
    First assume that $\gamma$ is differentiable at $q$. Then $\alpha_p + \alpha_m + \angle pqm = \pi$. When $p, m$ are close enough to $q$, we have 
    $$\frac{\pi}{2} < \angle_0 pqm  \le \angle qpr = \pi - (\alpha_p+\alpha_m)$$
    and 
    \begin{align*}
        \chi(p,q,m) &\ge \frac{2\sin (\alpha_p+\alpha_m)}{\sqrt{|qp|_X^2+ |pm|_X^2 + 2|qp|_X|pm|_X \cos (\alpha_p+\alpha_m)}} \\
        &\ge \frac{\sin (\alpha_p+\alpha_m)}{\sqrt{\alpha_p^2+\alpha_m^2 +2\alpha_p \alpha_m \cos (\alpha_p+\alpha_m)}} (\lambda-\epsilon)
    \end{align*}
    Since $\alpha_p, \alpha_m\to 0$ as $p,m\to q$, we have $\underline{\chi}_\gamma(q) \ge \lambda-\epsilon$ and (\ref{eqn unerline lambda le underline chi}) is shown by taking $\epsilon \to 0$. 

    Now assume that $\gamma$ has a corner at $q$ with turning angle $\theta$. Then for any $p, m\to q$, $\angle_0 pqm \le \pi-\theta<\pi$ and hence $\sin \angle _0 pqm \ge \sin \theta >0$ for all $p, m$ close to $q$. Hence $\underline \chi_\gamma (q) = \infty$ since $|pm|_X\to 0$ as $p, m\to q$. 
\end{proof}

\begin{prop}
 If $q$ is a regular boundary point of $\gamma$, then $\underline\chi_\gamma(q) \le \underline\lambda _\gamma (q)$ and $\overline \lambda_\gamma (q) \le \overline \chi_\gamma(q)$. 
\end{prop}

\begin{proof}
Let $p\in \partial X\setminus \{q\}$ be close to $q$, and let $m$ be any point close to $q$ on the opposite side of $\partial X$. Then 
$$ \chi (p,q,m) = \frac{2\sin \angle_0 pqm}{|pm|_X}.$$
As $m\to q$, the right hand side converges to 
$$ \frac{2\sin (\pi- \alpha)}{|pq|_X} = \frac{2\sin\alpha}{|pq|_X},$$
where $\alpha = \angle ([pq], \gamma'_\pm (q))$
\end{proof}






\section{Appendix C: Perimeter of the $\kappa$-lens $\mathcal L_{\kappa, \theta}$}
In this last appendix, we calculate the perimeter of the $\kappa$-lens.
\begin{prop}
The perimeter of $\partial \mathcal L_{\kappa, \theta}$ is
$$L(\partial \mathcal L_{\kappa, \theta}) =\frac{4}{\sqrt{K+\kappa^2}}\cdot\arcsin\sqrt{1-\frac{\kappa^2(1+\cos\theta)}{K(1-\cos\theta)+2\kappa^2}}.$$
\end{prop}
\begin{proof}
   We consider $K>0$ only as the case $K=0$ is simpler. Recall \(S_K=\{(x,y,z):x^2+y^2+z^2=1/K\} \subset \mathbb{R}^3\). Up to rigid motion of $S_K$, every $\kappa$-circle of $S_K$ is $O_1 = S_K\cap \{z=\frac{\kappa}{\sqrt{K}}\frac{1}{\sqrt{K+\kappa^2}}\}$, set \(R=\frac{1}{K+\kappa^2},\,a=\frac{\kappa}{\sqrt{K}}\frac{1}{\sqrt{K+\kappa^2}}\). Now, rotate $O_1$ around the \(x\)-axis in the opposite direction of the \(y\)-axis by an angle \(\phi\), resulting in another $\kappa$-circle \(O_2\),
 \[\begin{cases}
        x^2+(y+a\sin\phi)^2+(z-a\cos\phi)^2=R^2,\\
        z\cos\phi-y\sin\phi=a,
    \end{cases}\]
\(O_1\) and \(O_2\) intersect at two points $P_+, P_-$ which are 
\[\begin{split}
    P_\pm&=\left( \pm \sqrt{R^2-a^2\tan^2\left(\frac{\phi}{2}\right)},-a\tan\left(\frac{\phi}{2}\right),a\right).\end{split}\]
We assume that $\mathcal L_{\kappa, \theta} = O_1\cap O_2$. That is, the turning angles at $P_+, P_-$ are both $\theta$. By a direct calculation at $P_+$, we have 
$$\phi =2\arctan\sqrt{\frac{K(1-\cos\theta)}{K(1+\cos\theta)+2\kappa^2}}.$$ 
Let $\beta$ be the angle corresponding to the sector of $O_1$ with arc $\arc{P_1P_2}$. It's easy to see that \(\beta=2\arcsin\frac{\sqrt{R^2-a^2\tan^2\left(\frac{\phi}{2}\right)}}{R}\). Hence 
\[\begin{split}
L(\mathcal L_{\kappa, \theta}) &=4R\arcsin\sqrt{1-\frac{a^2}{R^2}\tan^2\left(\frac{\phi}{2}\right)}\\&=\frac{4}{\sqrt{K+\kappa^2}}\cdot\arcsin\sqrt{1-\frac{\kappa^2(1-\cos\theta)}{K(1+\cos\theta)+2g^2}}. \qedhere
\end{split}\]
\end{proof}

\bibliographystyle{amsplain}

\begin{thebibliography}{10}


\bibitem{AB} Alexander, S.; Bishop, R.:
{\sl Comparison theorems for curves of bounded geodesic curvature in metric spaces of curvature bounded above}. Differ. Geom. Appl. 6, No. 1, 67-86 (1996).

\bibitem{AB1} Alexander, S.; Bishop, R.:
{\sl Extrinsic curvature of semiconvex subspaces in Alexandrov geometry}. Ann. Global Anal. Geom. 37, No. 3, 241-262 (2010).

\bibitem{AKP}
Alexander, S.; Kapovitch, V.; Petrunin, A.:
{\sl Alexandrov geometry. Foundations}. Graduate Studies in Mathematics 236. Providence, RI: American Mathematical Society (AMS) (ISBN 978-1-4704-7302-0/hbk; 978-1-4704-7536-9/pbk; 978-1-4704-7535-2/ebook). xvii, 282 p. (2024).

\bibitem{AlexandrovselectedworksII} Alexandrov, A. D.; Kutateladze, S. S. (ed.):
{\sl Selected works. Intrinsic geometry of convex surfaces}. Vol. 2. Edited by S. S. Kutateladze. Transl. from the Russian by S. Vakhrameyev. Boca Raton, FL: Chapman \& Hall/CRC (ISBN 0-415-29802-4/hbk). xiii, 426 p. (2005).

\bibitem{BBI}
Burago, D.; Burago, Yu.; Ivanov, S.:
{\sl A course in metric geometry}. Graduate Studies in Mathematics. 33. Providence, RI: American Mathematical Society (AMS). xiv, 415 p. (2001).

\bibitem{Borisenko} Borisenko, A.:
{\sl An estimation of the length of a convex curve in two-dimensional Aleksandrov spaces}. J. Math. Phys. Anal. Geom. 16, No. 3, 221-227 (2020).


\bibitem{BGP} Burago, Yu.; Gromov, M. and Perelman, G., 
{\sl A. D. Aleksandrov spaces with curvatures bounded below}. Russian Math. Surveys
 47 (1992), no. 2, 1–58

\bibitem{DK}
Deng, Q.; Kapovitch, V.:
{\sl Sharp bounds and rigidity for volumes of boundaries of Alexandrov spaces}. arXiv:2308.14668

\bibitem{HW}
Hang, F.; Wang, X.:
{\sl Rigidity theorems for compact manifolds with boundary and positive
Ricci curvature}. J. Geom. Anal. 19 (2009), no. 3, 628–642.

\bibitem{GL1}
Ge, J.; Li, R.:
{\sl Rigidity for positively curved Alexandrov spaces with boundary}. Geom. Dedicata 213, 315-323 (2021).

\bibitem{GL2}
Ge, J.; Li, R.:
{\sl Radius estimates for Alexandrov space with boundary}. J. Geom. Anal. 31, No. 1, 619-630 (2021).

\bibitem{GP}
Grove, K.; Petersen, P.: 
{\sl Alexandrov spaces with maximal radius}. Geom. Topol. 26 (4) 1635 - 1668, 2022. https://doi.org/10.2140/gt.2022.26.1635

\bibitem{K}
Klingenberg. W.:{\sl Riemannian Geometry}. 2nd, rev. ed. Berlin; New York: de Gruyter, 1995 (De Gruyter studies in mathematics ; 1) ISBN 3-11-014593-6 NE: GT

\bibitem{P}
Petrunin, A.:
{\sl Semiconcave functions in Alexandrov’s geometry}. Cheeger, Jeffrey (ed.) et al., Metric and comparison geometry. Surveys in differential geometry. Vol. XI. Somerville, MA: International Press (ISBN 978-1-57146-117-9/hbk). Surveys in Differential Geometry 11, 137-201 (2007).

\bibitem{Toponogov}
Toponogov, V. A.:
{\sl Abschätzung der Länge einer konvexen Kurve}, Sibirsk. Mat. Zh.l 4 (1963), No. 5, 1189-1193.

 
\end{thebibliography}

\vspace{0.75cm}

    \noindent Le Ma\\
    Department of Mathematics, Southern University of Science and Technology \\
    No 1088, xueyuan Rd., Xili, Nanshan District, Shenzhen, Guangdong, China 518055 \\
    email: 12432025@mail.sustech.edu.cn

\bigskip

    \noindent John Man Shun Ma  \\
    Department of Mathematics, Southern University of Science and Technology \\
    No 1088, xueyuan Rd., Xili, Nanshan District, Shenzhen, Guangdong, China 518055 \\
    email: hunm@sustech.edu.cn

\end{document}